\documentclass[11pt]{article}
\usepackage{graphicx}
\usepackage{mathptmx, mathptmx, amsfonts}
\usepackage{psfrag}
\usepackage{float}
\usepackage{authblk}

\usepackage{geometry}
\usepackage[all]{xy}

\input xy
\xyoption{all}

\usepackage{graphics,marvosym}
\usepackage{tabularx}
\usepackage{caption}
\usepackage{subcaption}
\usepackage{rotating}
\usepackage{lscape}
\usepackage{amsthm}
\usepackage{multirow}

\newtheorem{theorem}{Theorem}

\usepackage{psfrag}

\usepackage{commath}
\usepackage{natbib}
\usepackage{dsfont}
\usepackage[unicode]{hyperref}
\hypersetup{pdftitle=MIMAR Paper,
pdfauthor={Sz. K\'alosi, S. Kapodistria, J. A. C. Resing},
pdftitle={CBM at  scheduled and unscheduled opportunities},
pdfstartview=FitH,
pdfpagemode=UseOutlines,
hidelinks
}
\usepackage{changes}
\colorlet{Changes@Color}{red}
\usepackage{caption}
\usepackage{subcaption}
\usepackage{cleveref}

\crefname{equation}{}{}

\usepackage{tikz}
\usetikzlibrary{automata,arrows,positioning,calc,decorations.pathreplacing,shapes}
\newcommand*{\Exp}{\ensuremath{\mathrm{Exp}}}

\newcommand{\pr}[1]{\ensuremath{\mathbb{P}\left( #1 \right)}}
\newtheorem{proposition}{\noindent \bf Proposition}

\title{Condition-based maintenance at both scheduled and unscheduled opportunities}
\author[*$\dag$]{Sz. K\'alosi }
\author[*$\S$]{S. Kapodistria }
\author[*$\ddag$]{J. A. C. Resing }
\affil[*]{Department of Mathematics and Computer Science, Eindhoven University of Technology, The~Netherlands,}
\affil[$\dag$]{email:  \texttt{s.kalosi@tue.nl}, ${}^{\S}${email:    \texttt{s.kapodistria@tue.nl}}, 
${}^{\ddag}${email:  \texttt{j.a.c.resing@tue.nl}}}

\date{\vspace{-7ex}}

\begin{document}

\maketitle

\begin{abstract}
Motivated by original equipment manufacturer (OEM) service and maintenance practices we consider a single component subject to replacements at failure instances and two types of preventive maintenance opportunities: scheduled, which occur due to periodic system reviews of the equipment, and unscheduled, which occur due to failures of other components in the system.
Modelling the state of the component appropriately and incorporating a realistic cost structure for corrective maintenance as well as condition-based maintenance  (CBM), we derive the optimal CBM  policy. In particular, we show that the optimal long-run average cost policy for the model at hand is a control-limit policy, where the control limit depends on the time until the next scheduled opportunity. Furthermore, we explicitly calculate the long-run average cost for any given control-limit time dependent policy and compare various policies numerically.
\end{abstract}

\section{Introduction }
When planning CBM strategies, see, e.g., \cite{Jardine2006}, \cite{jardine2013maintenance}, \cite{prajapati2012condition}, the typical assumption in the literature is that the asset at hand is monitored continuously and one can intervene and repair the  asset at any given moment. However, in practice, especially for highly complicated high-tech systems, such  as off-shore wind turbines, baggage handling systems in  airports and interventional X-ray machines, downtimes are extremely costly and scheduling maintenance tasks is challenging, see, e.g., \cite{kobbacy2008complex}, \cite{manzini2010maintenance}. One solution suggested in the literature to overcome this problem is the planning of opportunistic maintenance. Most of the works on opportunistic maintenance consider only unscheduled opportunities, rarely treating the case of scheduled opportunities, see, e.g., \cite{articlewang2002survey} and the references therein. Different from existing research on opportunistic maintenance policies, we consider both scheduled and unscheduled opportunities for CBM. Furthermore, we consider distinct costs for maintenance based on the type of the opportunity. For a single-component model we derive the optimal long-run average cost CBM policy and calculate the corresponding long-run average cost. Thereafter, we compare it with the corresponding cost for only one type of opportunity. Our work can be viewed as a variation of the work of \cite{zhuAge} in the direction of CBM.

The paper is organised as follows. In  \Cref{sec:modform}, we present the model and describe the exact setting that motivated our work. In the subsequent two sections we prove the two main results: in  \Cref{sec:polstruct}, we derive the optimal policy and in  \Cref{sec:longruncost}, we calculate the long-run average cost. In  \Cref{sec:numstud}, we numerically compare the optimal replacement policy to other policies. Finally, in the last section, we present some concluding remarks and ideas on future research.

\section{Model formulation}\label{sec:modform}
We consider a single component that is monitored continuously and whose condition is fully observable. We assume that the condition of the component can only degrade over time and it can be classified as {\em perfect}, {\em satisfactory} and {\em failed}. This type of models are known in the literature as  delay time degradation models, see \cite[Chapter~14]{kobbacy2008complex}. We will refer to the state of perfect condition as state $2$, the state of satisfactory condition as state $1$ and the failure state as state $0$. Furthermore, we assume that as soon as a component failure occurs, the component is instantaneously replaced by an ``as good as new'' component. So, in the mathematical formulation of the model, we may assume, due to the instantaneous replacement at failure, that the model evolves between only  states $1$ and $2$. The component spends an exponential amount of time with rate $\mu_i$ in state $i$, $i\in\{1,2\}$. The above model formulation implies that initially the component starts in state $2$ (perfect state), then after an exponential amount of time with rate $\mu_2$, the component deteriorates and the condition of the component goes to state $1$ (satisfactory state). The component spends an exponential amount of time with rate $\mu_1$ in state 1, after which a failure occurs. At a failure the component is instantaneously replaced by an ``as good as new'' component and the condition is restored to $2$ (perfect state). A schematic evolution of the condition of the component and the corresponding times of transitions are depicted in Figure \ref{Fig:1}.

\begin{figure}[h]
\[\def\labelstyle{\scriptstyle}
 \xymatrix  @C+9pt@R+1pt{
&
2  \ar@/^.75pc/[r]^{\mathrm{Exp}(\mu_2)}&1\ar@/^.75pc/[r]^{\mathrm{Exp}(\mu_1)}&2  \ar@/^.75pc/[r]^{\mathrm{Exp}(\mu_2)}&\cdots\ \ \ \\
}
\]
\caption{Schematic evolution of the condition of the component and the corresponding times of transitions.\label{Fig:1}}
\end{figure}
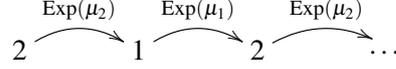

We assume that we have two types of opportunities in which we can preventively replace the component before failure: the  scheduled and the unscheduled opportunities. The scheduled opportunities correspond to pre-arranged opportunities occurring according to a fixed schedule. These opportunities can be attributed to either  service/maintenance agreements or to regulation imposition checks. We assume that the scheduled opportunities occur at epochs $\tau,2\tau,3\tau,\ldots$, with $\tau>0$. This is also in accordance with what happens in practice as maintenance actions once planned are typically not rescheduled. The unscheduled opportunities correspond to random opportunities triggered by failures of other unrelated components of the same system or failures of other systems in close proximity. We assume that these unscheduled opportunities occur according to a Poisson process at rate $\lambda$.

The unscheduled and scheduled opportunities, abbreviated by USO and SO, respectively, serve as opportunities for the monitored component to be replaced preventively. Such a preventive replacement is assumed to cost less than a corrective replacement upon failure, which costs say $c_c$. Moreover, incorporating a planning perspective, we may assume that the preventive replacement cost at a scheduled opportunity, say $c_p^{\text{so}}$, is less than or equal to the corresponding cost at an unscheduled opportunity, say $c_p^{\text{uso}}$, that is $0<c_p^{\text{so}}\le c_p^{\text{uso}}<c_c$.

Our aim is to determine a policy when to replace the component based on its condition and the opportunity type, i.e. scheduled or unscheduled. More explicitly, we will need to formally define the state space, which refers to the condition of the component, the action space and the decision epochs.
The state space is governed by the process depicting the condition of the component, i.e. the Markov chain evolving between the states $\{1,2\}$. The action space consists of only two actions: replace the component by an ``as good as new''  or do nothing. Lastly, the decision epochs are the epochs of the scheduled and unscheduled opportunities. In \Cref{Fig:2}, we depict SO by ($*$) and USO by (o).
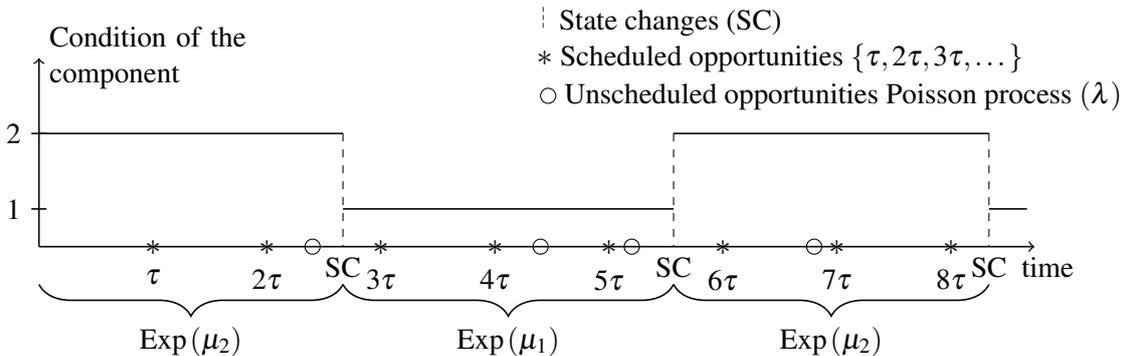
\begin{figure}[b!]
\begin{center}
		\begin{tikzpicture}
		\draw [semithick,->] (-6,0) -- (-6,2.5) node[right,text width=2.7cm]{Condition of the component};
		\draw [semithick] (-6.1 ,1.5) node[left]{2} -- (-5.9,1.5);
		\draw [semithick] (-6.1 ,.5) node[left]{1} -- (-5.9,.5);
		\draw [semithick,->] (-6,0) -- (7.1,0) node[below,xshift=1ex]{time};
		\draw [semithick] (-6,1.5) -- (-2,1.5);
		\draw  (-4.5,0) node[yshift=-.2pt]{\large $\ast$} node[below,yshift=-1ex]{ $\tau$};
		\draw  (-3,0) node[yshift=-.2pt]{\large $\ast$} node[below,yshift=-1ex]{ $2\tau$};
		\draw  (-2.4,0) circle (0.1);				
		\draw [semithick] (-2,.5) -- (2.35,.5);				
		\draw [dashed] (-2,1.5) -- (-2,0) node[below]{SC};								
		\draw  (-1.5,0) node[yshift=-.2pt]{\large $\ast$} node[below,yshift=-1ex]{ $3\tau$};
		\draw  (0,0) node[yshift=-.2pt]{\large $\ast$} node[below,yshift=-1ex]{ $4\tau$};
		\draw  (.6,0) circle (0.1);					
		\draw  (1.5,0) node[yshift=-.2pt]{\large $\ast$} node[below,yshift=-1ex]{ $5\tau$};				
		\draw  (1.8,0) circle (0.1);					
		\draw [dashed] (2.35,1.5) -- (2.35,0) node[below]{SC};				
		\draw [semithick] (2.35,1.5) -- (6.5,1.5);						
		\draw  (3,0) node[yshift=-.2pt]{\large $\ast$} node[below,yshift=-1ex]{ $6\tau$};		
		\draw  (4.2,0) circle (0.1);							
		\draw  (4.5,0) node[yshift=-.2pt]{\large $\ast$} node[below,yshift=-1ex]{ $7\tau$};				
		\draw  (6,0) node[yshift=-.2pt]{\large $\ast$} node[below,yshift=-1ex]{ $8\tau$};				
		\draw [dashed] (6.5,1.5) -- (6.5,0) node[below]{SC};	
		\draw [semithick] (6.5,.5) -- (7,.5);
		
		\draw[semithick,decorate,decoration={brace,amplitude=12pt,mirror}] 
		(-6,-.5) -- (-2,-.5)  node [midway,below,yshift=-12pt]  {$\Exp\del{\mu_2}$};
		\draw[semithick,decorate,decoration={brace,amplitude=12pt,mirror}] 
		(-2,-.5) -- (2.35,-.5)  node [midway,below,yshift=-12pt]  {$\Exp\del{\mu_1}$};				
		\draw[semithick,decorate,decoration={brace,amplitude=12pt,mirror}] 
		(2.35,-.5) -- (6.5,-.5)  node [midway,below,yshift=-12pt]  {$\Exp\del{\mu_2}$};
		\node at (3.75,2.5) {{\large{$\ast$}} Scheduled opportunities $\cbr{\tau,2\tau,3\tau,\dots}$};
		\draw  (.7,2) circle (0.1) node[right,xshift=.45ex]{Unscheduled opportunities Poisson process $(\lambda)$};					
		\node at (2.2,3) {{\large{\rotatebox{90}{\Kutline}}} State changes (SC)};
		\end{tikzpicture}
\end{center}
\caption{A sample path of the model.\label{Fig:2}}
\end{figure}

\section{Derivation of the optimal policy}\label{sec:polstruct}
This section is devoted to the proof of our first main result on the optimal policy on when to replace the component. To this purpose, we set up our problem as a Markov Decision Problem (MDP). Due to the stochastic nature of the problem, it does not suffice to know the type of the decision epoch (scheduled or unscheduled opportunity), but it is also required to keep track of the remainder time till the next scheduled opportunity. The remaining time until the next  scheduled opportunity may impact our decision, i.e. the optimal policy may depend on the residual time till the next scheduled opportunity. Thus, for the full description of the condition (state) of the component, we will use a triplet descriptor
$\mathcal{S}=\left\{(i,j,t):\ i\in \{1,2\}, \ j\in\{\text{SC},\text{SO},\text{USO}\},\ t\in[0,\tau)\right\}$, \label{page-ref}
where $i$ indicates the condition of the component. If $j=\text{SC}$, then this means that the condition of the system is about to change and there is no decision associated with this epoch, while if $j=\text{SO}$ or $j=\text{USO}$, this means that this is a decision moment of the type of a scheduled or unscheduled opportunity. Finally, the third element indicates the remaining time until the scheduled opportunity. Note that if  $j=\text{SO}$ then $t=0$. The introduction of the time in the full description of the condition of the component renders the model inhomogeneous, and for this reason we  use techniques that stem from semi-Markov Decision Problems (SMDP). \\

\subsection{Equal preventive costs}
In case of equal costs at scheduled and unscheduled opportunities,  the long-run optimal average cost policy does not depend on time. In particular, the optimal policy for state $2$ is to do nothing and for state $1$ the optimal policy depends on the model parameters.

\begin{theorem}\label{thm:opt1}
Under the assumption that $c_p^{\text{so}}=c_p^{\text{uso}}=c_p>0$,  the long-run optimal average cost policy is: For state $2$ to do nothing. For state $1$ to replace (at both scheduled and unscheduled opportunities) if $\del{\mu_1+\mu_2}c_p< \mu_1 c_c$, and to do nothing otherwise.
\end{theorem}

For the proof of Theorems \ref{thm:opt1} and \ref{thm:opt2}, we derive the policy that minimises the long-run average cost per time unit. To this purpose, we state the Bellman optimality equations for the decision process with action space $\mathcal{A}=\{\text{replace, do nothing}\}$, see, e.g., \cite{Puterman},  \cite{ross1970average}.

\begin{proposition}[Puterman, Section 11.4.]\label{thm:Puterman}
For a decision process with an embedded  unichain, there exists a scalar $g$ and a value function $V(\cdot)$ satisfying the system of equations
\begin{align}V(i,j,t)=\min_{a\in \mathcal{A}} \Big \{ c(i,j,t;a)-g \,s(i,j,t;a) +\sum_{k,l}\int_y p_{i,j,k,l}(t,y;a)V(k,l,y)\mathrm{d}y\Big\}, (i,j,t)\in \mathcal{S},\label{Bellman}
\end{align}
where $s(i,j,t;a)$ is the expected sojourn time of the SMDP in state $(i,j,t)$ with transition probabilities $p_{i,j,k,l}(t,y;a)$ under action $a\in \mathcal{A}$, and $c(i,j,t;a)$ is the stationary total expected cost of the SMDP in state $(i,j,t)$ between two consecutive decision epochs under action $a\in \mathcal{A}$. The existence of such a scalar $g$ ensures the existence of a policy that minimises the long-run average cost.
\end{proposition}

\noindent \begin{proof}[Proof of Theorem \ref{thm:opt1}]
The Bellman optimality equations \cref{Bellman} for the model at hand become:
\begin{align}
V(2,\text{SC},t) =& 0-g \int_{0}^{t} e^{-(\mu_1+\lambda)x}\dif x+ V(1,\text{SO},0)\int_{t}^{\infty}(\mu_1+\lambda)e^{-(\mu_1+\lambda)x}\dif  x \nonumber\\
&+ \int_{0}^{t} \left(\frac{\mu_1}{\mu_1+\lambda}V(1,\text{SC},t-x)+\frac{\lambda}{\mu_1+\lambda}V(1,\text{USO},t-x)\right) (\mu_1+\lambda)e^{-(\mu_1+\lambda)x}\dif x \label{eq:V2sca}\\
=& e^{-\del{\mu_1+\lambda} t} \del{ \int_{0}^{t}  \del{\mu_1 V(1,\text{SC},y)+\lambda V(1,\text{USO},y) - g}e^{\del{\mu_1+\lambda}y} \dif y + V(1,\text{SO},0)},\label{eq:V2sc}\\
V(1,\text{SC},t) =& c_c + e^{-\del{\mu_2+\lambda} t} \del{ \int_{0}^{t}  \del{\mu_2 V(2,\text{SC},y)+\lambda V(2,\text{USO},y) - g}e^{\del{\mu_2+\lambda}y} \dif y + V(2,\text{SO},0)},\label{eq:V1sc}\\
V(i,\text{USO},t) =& \min\left\{
c_p^{\text{uso}}+e^{-\del{\mu_2+\lambda} t} \del{ \int_{0}^{t}  \del{\mu_2 V(2,\text{SC},y)+\lambda V(2,\text{USO},y) - g}e^{\del{\mu_2+\lambda}y} \dif y + V(2,\text{SO},0)}
;
\right.\nonumber\\
&\left.\qquad
e^{-\del{\mu_i+\lambda} t} \del{ \int_{0}^{t}  \del{\mu_i V(i,\text{SC},y)+\lambda V(i,\text{USO},y) - g}e^{\del{\mu_i+\lambda}y} \dif y + V(i,\text{SO},0)}
\right\},\ i=1,2,\label{eq:USO}
\end{align}
\begin{align}
V(i,\text{SO},0) =& \min\left\{
c_p^{\text{so}}+
e^{-\del{\mu_2+\lambda} \tau} \del{\int_{0}^{\tau} \del{\mu_2 V(2,\text{SC},y)+\lambda V(2,\text{USO},y) - g}e^{\del{\mu_2+\lambda}y} \dif y + V(2,\text{SO},0)}
\right.;\nonumber\\
&\left.\qquad {}
e^{-\del{\mu_i+\lambda} \tau} \del{\int_{0}^{\tau} \del{\mu_i V(i,\text{SC},y)+\lambda V(i,\text{USO},y) - g}e^{\del{\mu_i+\lambda}y} \dif y + V(i,\text{SO},0)}
\right\},\ i=1,2.
\label{eq:SO}
\end{align}
In this paragraph we explain in detail how Equation \eqref{eq:V2sca} is obtained.
State $(2,\text{SC},t)$, as explained before, see page \pageref{page-ref}, is not associated with any decision. Therefore, there is no minimum operator appearing on the right hand side of the equation and the corresponding cost, $c(2,\text{SC},t;-)$, is  equal to zero. For the other terms appearing on the right hand side of equation \eqref{eq:V2sca}, it suffices to note that there are three possible evolutions in terms of the state of the system either a SO or a SC or  an USO, where the time till the next SO is equal to $t$, while the time till the SC and USO are exponentially distributed with rates $\mu_1$ and $\lambda$, respectively.   In particular, the expected sojourn time of the SMDP in state $(2,\text{SC},t)$ can be calculated as the minimum of a deterministic time $t$ and two exponentially distributed times, which can be easily verified to be equal to $\int_{0}^{t} e^{-(\mu_1+\lambda)x}\dif x$. 
The set of Equations \cref{eq:V1sc}--\cref{eq:SO} was obtained using very similar arguments. \\
Note that in \cref{eq:USO,eq:SO}, inside the minimum, the left terms correspond to the action ``replace'', while the right terms  correspond to the action ``do nothing''. Furthermore, for $i=2$,  \cref{eq:USO,eq:SO} yield that it is never optimal to replace. In order to decide if it is optimal or not to replace the component in state $i=1$ at scheduled and unscheduled opportunities, cf. \cref{eq:USO,eq:SO}, we define the following auxiliary functions, for $i=1,2,$
\begin{equation}\label{eq:fis}
F_i(t) = e^{-\del{\mu_{i}+\lambda} t} \del{ \int_{0}^{t}  \del{\mu_i V(i,\text{SC},y)+\lambda V(i,\text{USO},y) - g}e^{\del{\mu_i+\lambda}y} \dif y
+ V(i,\text{SO},0)}, \ t\in [0,\tau],
\end{equation}
and rewrite \cref{eq:V2sc,eq:V1sc,eq:USO,eq:SO}, for $t\in[0,\tau)$,
\begin{align}
V(1,\text{SC},t) =& c_c + F_2(t),\quad V(2,\text{SC},t) = F_1(t), \label{V2SCt}\\
V(i,\text{USO},t) =& \min\cbr{F_i(t),c_p^{\text{uso}}+F_2(t)}, \ i=1,2,\label{ViUSOt}\\
V(i,\text{SO},0) =& \min\cbr{F_i(\tau),c_p^{\text{so}}+F_2(\tau)},\ i=1,2 .\label{ViSOt}
\end{align}
From this point onward in the proof, we will use that $c_p^{\text{uso}}=c_p^{\text{so}}=c_p$. At a scheduled opportunity, we can either replace a component in state $1$ or do nothing. We first assume that we replace the component at a scheduled opportunity, i.e.
$
F_1(\tau) > c_p+F_2(\tau),
$
thus, Equation \cref{ViSOt} for $i=1$ implies that
$V(1,\text{SO},0)=c_p+F_2(\tau)$.
Furthermore, for $t=0$, Equation \cref{eq:fis} yields $F_1(0)=V(1,\text{SO},0)$ and $F_2(0)=V(2,\text{SO},0)\stackrel{\eqref{ViSOt}}{=}F_2(\tau)$. Combining the last two results yields
\begin{align}\label{asumption1}
F_1(0)-F_2(0)=c_p < F_1(\tau)-F_2(\tau).
\end{align}
Since the functions $F_i(t)$ are continuous functions in $t\in[0,\tau]$ and their difference at $t=\tau$ is greater than
$c_p$ at $\tau$, cf. Equation \cref{asumption1}, there exists an $\varepsilon>0$ such that
\begin{align}\label{asumption2}
F_1(t)-F_2(t) \geq  c_p,\ t\in(\tau-\epsilon,\tau].
\end{align}
Taking the derivative in \cref{eq:fis} and after straightforward manipulations, using  \cref{V2SCt}--\Cref{ViSOt} together with  \Cref{asumption2},
 yields
\begin{align}
F_1'(t)-F_2'(t)=&-(\mu_1+\lambda)F_1(t)+\mu_1 V(1,\text{SC},t)+\lambda V(1,\text{USO},t)\nonumber\\
& +(\mu_2+\lambda)F_2(t)-\mu_2 V(2,\text{SC},t)-\lambda V(2,\text{USO},t)\nonumber\\
=&
-(\mu_1+\mu_2+\lambda )(F_1(t)-F_2(t)) +\lambda c_p+\mu_1 c_c,\ t\in(\tau-\epsilon,\tau].
\label{eq:difeqs}
\end{align}
The solution to the above differential equation  reads
\begin{align}
F_1(t)-F_2(t) =& (F_1(\tau)-F_2(\tau))e^{(\mu_1+\mu_2+\lambda)(\tau-t)}-(\lambda c_p+\mu_1 c_c)\int_t^\tau e^{(\mu_1+\mu_2+\lambda)(x-t)}\dif x \nonumber\\
=&\frac{\lambda c_p+\mu_1 c_c}{\mu_1+\mu_2+\lambda}+\left(F_1(\tau)-F_2(\tau)-\frac{\lambda c_p+\mu_1 c_c}{\mu_1+\mu_2+\lambda}\right)e^{(\mu_1+\mu_2+\lambda)(\tau-t)},\ t\in(\tau-\epsilon,\tau].
\label{solution:dif}
\end{align}
Note that, for $t\in(\tau-\epsilon,\tau]$ the function $F_1(t)-F_2(t)$ is monotone. Moreover, if $F_1(\tau)-F_2(\tau)-\frac{\lambda c_p+\mu_1 c_c}{\mu_1+\mu_2+\lambda}\geq 0$ or equivalently if $F_1(t)-F_2(t)$ is non-increasing,  we can extend the above approach throughout the interval $[0,\tau]$ and show that the solution of the differential equation \Cref{solution:dif} is valid in the entire interval $[0,\tau]$, but that would contradict assumption \Cref{asumption1}. Hence, for $t\in(\tau-\epsilon,\tau]$ the function $F_1(t)-F_2(t)$ is increasing and
\begin{align}\label{condition:1}
c_p < F_1(\tau)-F_2(\tau) < \frac{\lambda c_p+\mu_1 c_c}{\mu_1+\mu_2+\lambda}.
\end{align}
\begin{figure}[h!]
\psfrag{a}[B][c][1][90]{$F_1(t)-F_2(t)$}
\psfrag{b}[r][r]{$c_p=c_p^{\text{so}}=c_p^{\text{uso}}$}
\psfrag{c}[c][c]{$\tau-\varepsilon$}
\psfrag{d}[c][c]{$\tau$}
\psfrag{t}[c][c]{$t$}
\centerline{\includegraphics[width=0.8\textwidth]{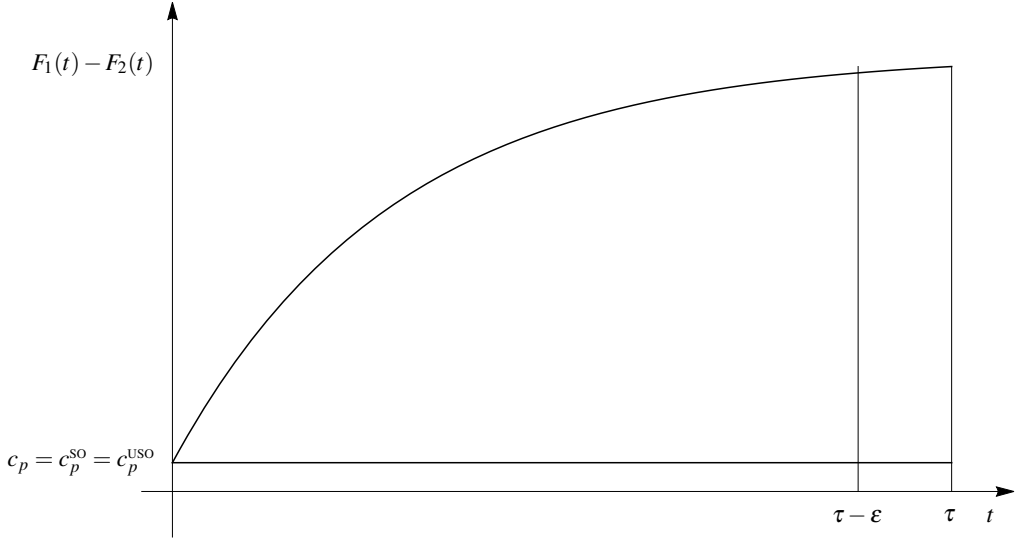}}
\caption{The case of replacing a component at scheduled and unscheduled opportunities in $[0,\tau)$.}\label{Fig:3}
\end{figure}

\noindent Since the function $F_1(t)-F_2(t)$ is increasing for $t\in(\tau-\epsilon,\tau]$, we now need to identify the $\epsilon$, more concretely the point $t=\tau-\epsilon\geq 0$ such that $F_1(t)-F_2(t)= c_p$. If $t=\tau-\epsilon$ is strictly positive this contradicts assumption \Cref{asumption1}, yielding $\epsilon=\tau$. All in all, if the policy at a scheduled opportunity is to replace the component in state 1, then the optimal long-run average cost policy at an unscheduled opportunity is to replace in state $1$ if $\del{\mu_1+\mu_2}c_p< \mu_1 c_c$, cf. \Cref{condition:1}. A schematic representation of the function $F_1(t)-F_2(t)$ is depicted in Figure \ref{Fig:3}.

Similarly to the case that we replace the component at a scheduled opportunity, we may assume that we do not replace the component at a scheduled opportunity and prove that then, it is also optimal not to replace the component at an unscheduled opportunity if $\del{\mu_1+\mu_2}c_p> \mu_1 c_c$. At the equality $\del{\mu_1+\mu_2}c_p =  \mu_1 c_c$, we are cost-wise indifferent between the two actions.
\end{proof}
\medskip

\subsection{Different preventive costs}
 In case of different preventive costs at scheduled and unscheduled opportunities, the optimal policy, in certain cases, becomes time-dependent, i.e. there exists a threshold $t^*$ such that it is optimal to replace if the residual time until the next scheduled opportunity is greater than or equal to the threshold and do nothing otherwise.

\begin{theorem}\label{thm:opt2}
Under the assumption that $c_p^{\text{so}}< c_p^{\text{uso}}$,  the optimal long-run average cost policy is: For state $2$ to do nothing. For state $1$ to replace at scheduled opportunities and at unscheduled opportunities occurring in $[t^*,\tau)$, with $ t^*=\max\left\{0,
\frac{1}{\mu_1+\mu_2}\ln\del{\frac{\del{\mu_1+\mu_2}c_p^{\text{so}}-\mu_1 c_c}{\del{\mu_1+\mu_2}c_p^{\text{uso}}-\mu_1 c_c}}\right\}$,  if $\del{\mu_1+\mu_2}c_p^{\text{so}}<\mu_1 c_c$, and to do nothing otherwise.
\end{theorem}
\noindent \begin{proof}
Similarly to the proof of \Cref{thm:opt1}, we need to make certain assumptions here regarding the actions at the given opportunities. In particular, we distinguish three cases, each corresponding to a different set of actions.
{\em Case (i):} $F_1(\tau)-F_2(\tau)> c_p^{\text{uso}}$; {\em Case (ii):} $c_p^{\text{uso}}\geq F_1(\tau)-F_2(\tau)> c_p^{\text{so}}$; {\em Case (iii):} $F_1(\tau)-F_2(\tau)\leq c_p^{\text{so}}$.
The proof of this theorem is identical in structure to the proof of Theorem \ref{thm:opt1} and for this reason it is omitted.

Schematic representations of the function $F_1(t)-F_2(t)$ for cases (i) and (ii) are depicted in Figures \ref{Fig:4a} and \ref{Fig:4b}, respectively.
\begin{figure}[h!]
    \centering
    \begin{subfigure}[b]{0.48\textwidth}
       \includegraphics[width=\textwidth]{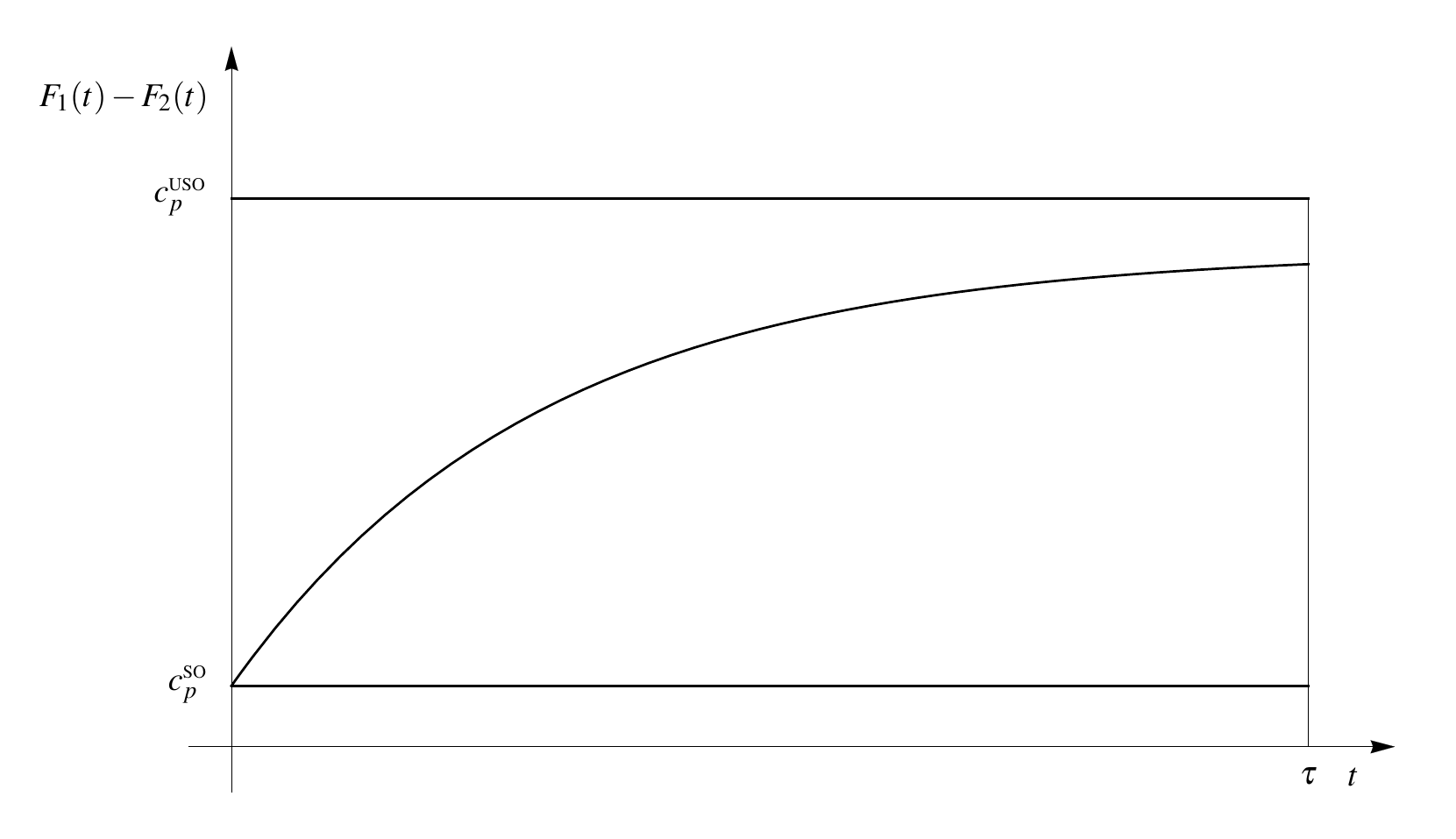}
        \caption{Case (i): The case of replacing a component at \\ scheduled  and  unscheduled opportunities in $[t^*,\tau)$.}
		\label{Fig:4a}
    \end{subfigure}\!\!
    \begin{subfigure}[b]{0.48\textwidth}
        \includegraphics[width=\textwidth]{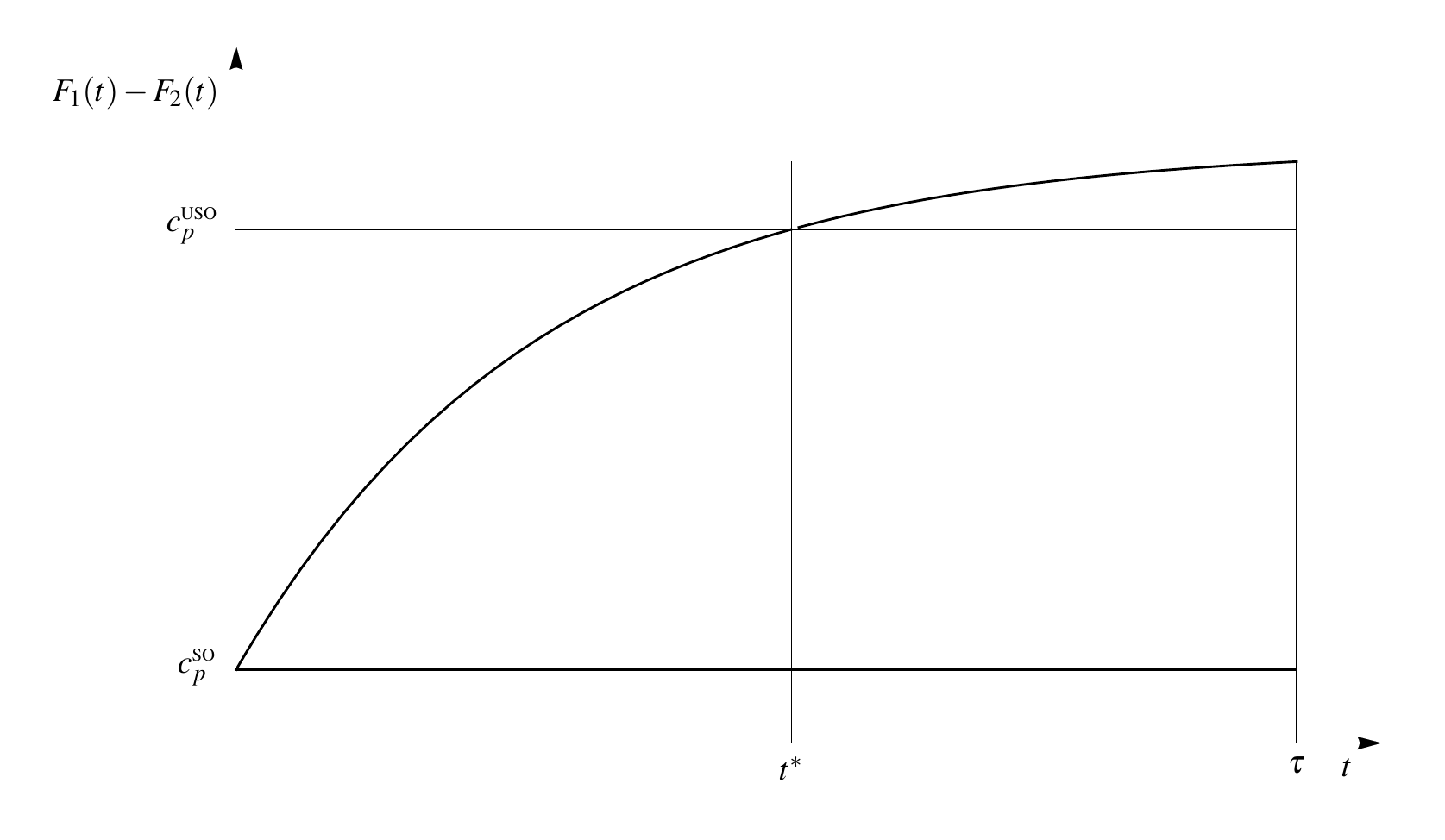}
        \caption{Case (ii): The case of replacing a component at only scheduled opportunities.}
        \label{Fig:4b}
    \end{subfigure}
\caption{Cases (i) and (ii) of Theorem \ref{thm:opt2}.}
\end{figure}\
\end{proof}

\section{Derivation of the long-run average cost}\label{sec:longruncost}
In the previous section, we derived the optimal long-run average cost policy. In this section, for any given time dependent control limit policy, we calculate the  long-run average cost per time unit  in Theorem \ref{them:costs}.

\begin{theorem}\label{them:costs}
Consider a given policy that states that in state $2$ we do nothing, and in state $1$ we replace at scheduled opportunities and at unscheduled opportunities for which the remaining time until the next scheduled opportunity is greater than $\tilde{t}\in(0,\tau)$,
and we do nothing otherwise. Under this given policy, the long-run average cost is
\begin{align}
&\frac{1}{\tau}\left[
c_c\mu_1  \del{\tilde{t}\frac{ \mu_2}{\mu_1+\mu_2}+\del{\frac{\mu_2}{\lambda+\mu_1+\mu_2}-\frac{\mu_2}{\mu_1+\mu_2} - \frac{\mu_2}{\lambda+\mu_1+\mu_2} e^{\del{\lambda+\mu_1+\mu_2}(\tilde{t}-\tau)}}\frac{1}{\mu_1+\mu_2}\del{1-e^{-(\mu_1+\mu_2)\tilde{t}}}}\right.
\nonumber\\
&\quad +\del{c_{p}^{USO}\lambda+c_{c}\mu_1}\frac{\mu_2}{\lambda+\mu_1+\mu_2}\del{\tau-\tilde{t}-\frac{1}{\lambda+\mu_1+\mu_2}\del{1-e^{\del{\lambda+\mu_1+\mu_2}\del{\tilde{t}-\tau}}}} \nonumber\\
&\quad \left.
+c_p^{\text{so}}\del{\frac{\mu_2}{\mu_1+\mu_2} + \del{\frac{\mu_2}{\lambda+\mu_1+\mu_2}-\frac{\mu_2}{\mu_1+\mu_2}-\frac{\mu_2}{\lambda+\mu_1+\mu_2}e^{\del{\lambda+\mu_1+\mu_2}(\tilde{t}-\tau)}}e^{-(\mu_1+\mu_2)\tilde{t}}}
\right].
\label{eq:replaceT}
\end{align}
\end{theorem}

\noindent \begin{proof}
For the calculation of the long-run average cost, we employ techniques from renewal-reward theory. To this purpose, we consider as renewal epochs the instants of the scheduled opportunities, thus, the inter-renewal times are deterministic and equal to $\tau$. For the determination of the average cost during a renewal cycle, we define the following probabilities 
\begin{equation}\label{eq:propbsilam}
p_i\del{t}=\pr{\text{machine is in state }i\mid\text{time until next scheduled opportunity is }t },
\end{equation}
where $t\in[0,\tau)$. Then, the average cost in a cycle can be obtained as 
\begin{equation*}
c_p^{\text{so}}p_1(0)+\del{c_p^{\text{uso}}\lambda+c_c \mu_1}\int_{\tilde{t}}^{\tau}p_1(t)\dif t + c_c \mu_1 \int_{0}^{\tilde{t}} p_1(t) \dif t,
\end{equation*}
from which, once we calculate the unknown probability $p_1\del{t}$, $t\in[0,\tau)$, we can immediately derive the long-run average cost by simply dividing by  $\tau$. \\
The probabilities in \cref{eq:propbsilam} satisfy the following system of differential equations
\begin{align}
p_1'\del{t} =& \mu_1 p_1\del{t} - \mu_2 p_2\del{t}\ \text{ and }\ p_2'\del{t} = \mu_2 p_2\del{t} -\mu_1 p_1\del{t}, \ t\in[0,\tilde{t})\label{eq:p1>t(1)}\\
p_1'\del{t} =& \del{\lambda+\mu_1} p_1\del{t} - \mu_2 p_2\del{t}\ \text{ and } \ 
p_2'\del{t} = \mu_2 p_2\del{t} - \del{\lambda+\mu_1} p_1\del{t}, \ t\in[\tilde{t},\tau). \label{eq:p1>t(2)}
\end{align}
Solving \eqref{eq:p1>t(1)}-\eqref{eq:p1>t(2)}, knowing that $p_1(t)+p_2(t)=1$ for all $t$,  $\lim\limits_{t\to\tau^-}p_1(t)=0$ and $\lim\limits_{t\to\tilde{t}^-}p_i(t)=\lim\limits_{t\to\tilde{t}^+}p_i(t)$, $i=1,2$, i.e. they are continuous functions, we obtain
\begin{align*}
p_1(t)=&\frac{\mu_2}{\mu_1+\mu_2} + \del{\frac{\mu_2}{\lambda+\mu_1+\mu_2}-\frac{\mu_2}{\mu_1+\mu_2}-\frac{\mu_2}{\lambda+\mu_1+\mu_2}e^{\del{\lambda+\mu_1+\mu_2}(\tilde{t}-\tau+\frac{\Lambda(t)}{\lambda}(t-\tilde{t}))}}e^{\frac{\lambda-\Lambda(t)}{\lambda}(\mu_1+\mu_2)(t-\tilde{t})},
\end{align*}
where
\begin{align*}
\Lambda(t)=\begin{cases}
0,&\text{ if } \ 0\leq t<\tilde{t},\\
\lambda,&\text{ if }\ \tilde{t}\leq t<\tau.
\end{cases}
\end{align*}
\end{proof}

\subsection{Special cases}
In case of only {\em scheduled opportunities}, which corresponds to the case  $\tilde{t}\to\tau$ or equivalently  to the case $\lambda\to0$, the long-run expected cost equals $$\frac{1-e^{-(\mu_1+\mu_2)\tau}}{(\mu_1+\mu_2)\tau}\mu_2\del{c_p^{\text{so}}-c_c\frac{\mu_1}{\mu_1+\mu_2}}+c_c\frac{\mu_1\mu_2}{\mu_1+\mu_2}.$$

In case of only {\em unscheduled opportunities}, which corresponds to the case  $\tau\to\infty$, the condition of the component can be fully described using a double descriptor
$\mathcal{S}=\left\{(i,j):\ i\in \{1,2\}, \ j\in\{\text{SC},\text{USO}\}\right\}$ which is independent of time, and thus the new model formulation falls into the framework of regular MDPs. It can be easily shown that:
For state $2$ the optimal policy is to do nothing. For state 1 the optimal policy is  to replace if
$(\mu_1+\mu_2)c_p^{\text{uso}} < \mu_1 c_c$
and to do nothing otherwise. Furthermore, under the optimal policy the  average long-run cost is equal to
 $$ \frac{c_p^{\text{uso}}\lambda\mu_2+c_c\mu_1\mu_2}{\lambda+\mu_1+\mu_2}.$$

In case of only {\em corrective replacements}, the long-run average cost is equal to
 $$c_c\frac{\mu_1\mu_2}{\mu_2+\mu_1}.$$

\section{Numerical results}\label{sec:numstud}
As a small numerical example, we examine the effect of $\tau$, $\lambda$ and $c_p^{\text{so}}$ on the long-run average cost, while keeping all other parameters fixed. We choose $c_c=15000$, $c_p^{\text{uso}}=10000$, $\mu_2=.4$ and $\mu_1=1$. In this case, the long-run average cost in case of only corrective replacements is equal to $4285.71$. In Table 1, we show the long-run average cost for various choices of $\lambda$, $\tau$ and $c_p^{\text{so}}$ for three different policies:
the optimal policy ($a_{\text{opt}}$), the policy that we replace at only the scheduled opportunities ($a_{\text{so}}$), and finally,
the policy under which we always replace at all opportunities ($a_{\text{alw}}$). It seems that the use of the proposed policies can considerably improve upon the long-term average cost when compared to performing only corrective replacements. However, in most of the numerical scenarios we performed, as long as the choice of the parameters is reasonable, we did not notice considerable differences between the three policies, viz. $a_{\text{opt}},\ a_{\text{so}},\ a_{\text{alw}}$.

\begin{table}[h]
	\centering
	 \begin{tabular}{|c|cc|ccc|ccc|}
			\hline  \multicolumn{9}{|c|}{$c_p^{SD}=4000$ and $t^*\approx1.6$}
			\\ \hline			
			 \multicolumn{3}{|c|}{$\tau=1$} & \multicolumn{3}{c|}{$\tau=2$} & \multicolumn{3}{c|}{$\tau=4$}    \\  \hline
			$\lambda$ & $a_{\text{opt}}=a_{\text{so}}$ & $a_{\text{alw}}$ & $a_{\text{opt}}$ & $a_{\text{so}} $& $a_{\text{alw}}$ & $a_{\text{opt}}$ & $a_{\text{so}}$& $a_{\text{alw}}$  \\ \hline
			0.1 & 2840.41 & 2885.56 & 3384.70 & 3384.86 & 3422.03 & 3802.49 & 3807.90 & 3823.32 \\
			0.5 & 2840.41 & 3042.07 & 3384.09 & 3384.86 & 3538.91 & 3784.63 & 3807.90 & 3867.21 \\			
			1 & 2840.41 & 3194.24 & 3383.38	 & 3384.86 & 3636.35 & 3768.42 & 3807.90 & 3899.32 \\			
			2 & 2840.41 & 3401.88 & 3382.15 & 3384.86 & 3747.82 & 3747.68 & 3807.90 & 3932.53 \\ \hline
			 \multicolumn{9}{|c|}{$c_p^{SD}=6500$ and $t^*\approx1.27$}  \\ \hline
			\multicolumn{3}{|c|}{$\tau=1$} & \multicolumn{3}{c|}{$\tau=2$} & \multicolumn{3}{c|}{$\tau=4$}    \\  \hline
			$\lambda$ & $a_{\text{opt}}=a_{\text{so}}$ & $a_{\text{alw}}$ & $a_{\text{opt}}$ & $a_{\text{so}} $& $a_{\text{alw}}$ & $a_{\text{opt}}$ & $a_{\text{so}} $& $a_{\text{alw}}$  \\ \hline
			0.1  & 3378.56 & 3403.48 & 3719.49 & 3720.28 & 3738.77 & 3979.00 & 3985.81 & 3989.58 \\
			0.5 & 3378.56 & 3489.66 & 3716.57 & 3720.28 & 3796.18 & 3956.81 & 3985.81 & 3998.72 \\
			1 & 3378.56 & 3573.11 & 3713.40 & 3720.28 & 3842.96 & 3937.06 & 3985.81 & 4003.48 \\
			2 & 3378.56 & 3686.18 & 3708.32 & 3720.28 & 3894.71 & 3912.27 & 3985.81 & 4006.06 \\ \hline
			 \multicolumn{9}{|c|}{$c_p^{SD}=9000$ and $t^*\approx0.63$}   \\ \hline
			\multicolumn{3}{|c|}{$\tau=1$} & \multicolumn{3}{c|}{$\tau=2$} & \multicolumn{3}{c|}{$\tau=4$}    \\  \hline
			$\lambda$ & $a_{\text{opt}}=a_{\text{so}}$ & $a_{\text{alw}}$ & $a_{\text{opt}}$ & $a_{\text{so}} $& $a_{\text{alw}}$ & $a_{\text{opt}}$ & $a_{\text{so}} $& $a_{\text{alw}}$  \\ \hline
			0.1 & 3792.57 & 3797.66 & 3916.43 & 3916.70 & 3921.39 & 4052.18 & 4055.71 & 4055.51  \\
			0.5 & 3792.57 & 3816.41 & 3915.40 & 3916.70 & 3937.26 & 4039.84 & 4055.71 & 4053.46   \\
			1 & 3792.57 & 3836.68 & 3914.21 & 3916.70 & 3951.98 & 4027.59 & 4055.71 & 4049.58   \\
			2 & 3792.57 & 3868.78 & 3912.11 & 3916.70 & 3970.48 & 4010.20 & 4055.71 & 4041.61  \\ \hline		
		\end{tabular}
		\caption{{Long-run average cost varying  $\lambda$, $\tau$ and $c_p^{\text{so}}$, while keeping  all other parameters fixed for three policies.}}
		\label{table:longrun}		
	\end{table}

\section{Conclusions and future work}\label{sec:concl}
In this paper, we derived the optimal replacement policy for a 3-state component degrading over time with corrective replacements at failures and preventive replacements at both scheduled and unscheduled opportunities. We also explicitly calculated the long-run average cost for certain time-dependent policies. In future work, we want to extend our analysis to models of components with more than three states, to models with additional features such as repairs, and to models in which the condition of the component is only partially observable.

\section*{Acknowledgments}
\noindent The work of Sz. K\'alosi is supported by the Data Science Flagship framework, a cooperation between the Eindhoven University of Technology and Philips.
The work of S. Kapodistria is supported by an NWO Gravitation Project, NETWORKS, and a TKI Project, DAISY4OFFSHORE. The authors would like to thank M. Barbieri, J. Korst, and V. Pronk (all Philips Research) and O. J. Boxma (Eindhoven University of Technology) for their time and advice in the preparation of this work.

\end{document}